\documentclass{amsart}
\usepackage{amscd,amsmath}
\usepackage[dvips]{graphicx}
\usepackage{hyperref}
\usepackage[latin1]{inputenc}
\usepackage{array}
\usepackage[all]{xy}

\newcommand{\sing}{{\rm{Sing}}}

\newtheorem{theorem}{\bf Theorem}
\newtheorem{proposition}{\bf Proposition}

\newtheorem{remark}{Remark}

\newtheorem{example}{\bf Example}

\begin{document}

\title[Logarithmic Hesse's problem]{Logarithmic Hesse's problem}

\author{T. FASSARELLA}
\address{Departamento de Análise -- IM \\
UFF \\
Mário Santos Braga S/N -- Niterói\\
24.020-140 RJ Brasil} \email{thiago@impa.br}


\begin{abstract}
We show that if a Laurent polynomial on the coordinate ring of the complex algebraic torus on $n$ variables has vanishing logarithmic  Hessian, then up to an automorphism of the torus, the Laurent polynomial depends on at most $n-1$ variables.  
\end{abstract}

\maketitle

\section{Introduction}
Let $f\in\mathbb C[x_1,...,x_n]$, $n\ge 1$, be a polynomial and consider the \textit{Hessian matrix} 
\begin{eqnarray*}
Hf(x):=\left(f_{x_ix_j}(x)\right)_{1\le i,j\le n}.
\end{eqnarray*}
The polynomial $f$ has \textit{vanishing Hessian} if the determinant ${\rm{det}}Hf \in \mathbb C[x_1,...,x_n]$ is null.  This is equivalent to the condition on the derivatives of $f$ being algebraically  dependent, that is, $f$ has vanishing Hessian if and only if the \textit{affine polar map}
\begin{eqnarray*} 
\nabla_f:\mathbb C^n &\longrightarrow& \mathbb C^n\\
             x &\mapsto& \left( f_{x_1}(x),...,f_{x_n}(x)\right)
\end{eqnarray*}
is not dominant. 

Hesse claimed in \cite{H1,H2} that a homogeneous polynomial on $n$ variables has vanishing Hessian if and only if after a suitable coordinate change, the polynomial depends on at most $n-1$ variables. While the "if" implication is trivial, the "only if" statement is not trivial at all. The Hesse problem was taken up by Gordan--Noether in \cite{GN}, who showed that the question has an affirmative answer for $n \le 4$, but is false in general for $n \ge 5$. See \cite{permutti1957certe, permutti1963sul, permutti1964certe, lossen2004does} for a review of the counterexamples constructed by Gordan--Noether. See also \cite{CRS} for a modern overview of the known methods to deal with the problem. The case of nonhomogeneous polynomials was studied in \cite{BV} where a classification for small $n$ was obtained. 

In this paper we will consider polynomials $f\in S_{n}=\mathbb C[x_1,x_1^{-1},...,x_n,x_n^{-1}]$, that is, in the coordinate ring  of the torus $(\mathbb C^*)^n$. An element $f \in S_{n}$ is called a \textit{Laurent polynomial} on $n$ variables. In this case we shall work with the \textit{logarithmic polar map} associated to  $f$ defined as follows 
\begin{eqnarray*}
L_f:(\mathbb C^*)^n &\longrightarrow& \mathbb C^n\\
             x &\mapsto& \left( x_1f_{x_1}(x),...,x_nf_{x_n}(x)\right).
\end{eqnarray*}

As in the affine case, we note that $L_f$ can be seen as the logarithmic Gauss map associated to the hypersurface $Z=\{f=0\}\subset (\mathbb C^*)^n$, that is, the map that takes $x\in Z$ to the left translation to unity  of the hyperplane $T_xZ\subset T_x(\mathbb C^*)^n$ (see example \ref{Ex:gauss}). 

We consider now the  matrix
\begin{eqnarray*}
Af(x):= \left((x_i\cdot f_{x_i})_{x_j} (x)\right)_{i,j=1,...,n},
\end{eqnarray*}
which will be called \textit{logarithmic Hessian matrix} of $f$. We shall say that $f$ has \textit{vanishing logarithmic Hessian} if the determinant ${\rm{det}}Af \in S_{n}$  is null.  If $f$ is a Laurent polynomial, the condition of vanishing logarithmic Hessian is equivalent to the non--dominance hypothesis on $L_f$. 

The main objective of this paper is to show that -- as a counterpart to Hesse's problem -- in the case of Laurent polynomials on the torus, the analogue of Hesse's claim is true. That is, if  a Laurent polynomial on $n$ variables has vanishing logarithmic Hessian then after an automorphism of the torus, the Laurent polynomial depends on at most $n-1$ variables. This is  Theorem \ref{T:Hesse}. Before that we show that the fibers of $L_f$ have a good behavior, more precisely, under the hypothesis of vanishing logarithmic Hessian, the logarithmic Gauss map associated to the foliation determined by its fibers is constant. This is Proposition \ref{P:Hesse}.

\section{Fibers of the affine polar Map}

In order to study the fibers of the logarithmic polar map we first restrict our attention to the fibers of the affine polar map. Let $f$ be a holomorphic map defined on an open subset $U$ of $\mathbb C^n$. Taking the derivatives of $f$ we can define, as in the polynomial case, the affine polar map 
\begin{eqnarray*}
\nabla_f:U &\longrightarrow& \mathbb C^n\\
    x &\mapsto& (f_{x_1}(x),...,f_{x_n}(x)). 
\end{eqnarray*}
In the following proposition we use the same idea as in \cite[Appendix]{FW} to study its fibers. 

\begin{proposition}\label{T:linearfiber}
If the affine polar map $\nabla_f:U \longrightarrow \mathbb C^n$ has constant rank $n-k<n$, then a fiber of $\nabla_f$ is an union of open subsets of affine linear subspaces of dimension $k$ of $\mathbb C^n$. 
\end{proposition}

\begin{proof}
We set $r=n-k$. Since the problem is local on $U$, one can take a parametrization $\varphi:S\times T \longrightarrow U$, where $S\subset \mathbb C^k$, $T\subset \mathbb C^r$  are open sets and $S\times \{t \}$ are fibers of $\Gamma:=\nabla_f \circ \varphi$. 

We have to show that $\varphi(S\times \{t \})$ is an open subset of some affine linear subspace on $\mathbb C^n$. Denote by $z=(s,t)\in S\times T$ and $\varphi_{s_i}$, $\varphi_{t_j}$, $i=1,...,k$, $j=1,...,r$, the natural derivatives. We will show that  $\varphi(S\times \{t \})\subset F_z$ where
\begin{eqnarray*}
F_{z}=\textrm{span}\{\varphi_{s_1}(z),...,\varphi_{s_k}(z)\}=\textrm{ker}d\nabla_f(x). 
\end{eqnarray*}

We first note that $F_z=\tilde{F_z}$, where
\begin{eqnarray*}
\tilde{F_z}=\{ V \in \mathbb C^n \;|\; <\Gamma_{t_j}(z),V>=0 \;,j=1,...,r\}.
\end{eqnarray*}
Since $d\nabla_f(x)$ self-adjoint and $\Gamma_{t_j}=d\nabla_f\cdot\varphi_{t_j}$ we get
\begin{eqnarray}\label{Eq:adjoint}
<\Gamma_{t_j},\varphi_{s_i}>=<\varphi_{t_j},d\nabla_f\cdot\varphi_{s_i}>=0.
\end{eqnarray}
Hence $F_z\subset \tilde{F_z}$. But $\dim F_z=\dim \tilde{F_z}$, therefore $F_z=\tilde{F_z}$.

Now we claim that $\varphi_{s_is_j}(z)\in \tilde{F_z}$. In fact, it follows from $\Gamma_{s_l}\equiv 0$ for all $l=1,...,k$, and (\ref{Eq:adjoint}) that
\begin{eqnarray*}
<\Gamma_{t_l},\varphi_{s_is_j}>=<\Gamma_{t_l},\varphi_{s_is_j}>+
<\Gamma_{t_ls_i},\varphi_{s_j}>=<\Gamma_{t_l},\varphi_{s_j}>_{s_i}=0.
\end{eqnarray*}
Therefore $\varphi_{s_is_j}(z)\in F_z$, that is, $\varphi$ satisfies the following system of PDE:
\begin{eqnarray}\label{Eq:PDE}
\displaystyle{\varphi_{s_is_j}=\sum_{l=1}^{k}\mu_{ijl}\cdot\varphi_{s_l}},
\end{eqnarray}
for some holomorphic functions $\mu_{ijl}$.

By a linear coordinate change, we may assume that $z=(0,0)$ and 
\begin{eqnarray}\label{Eq:F}
F_z=\{x\in\mathbb C^n \;|\;x_{k+1}=...=x_n=0\}. 
\end{eqnarray} 

In order to show that $\varphi(S\times \{0 \})\subset F_z$, we have to prove that
\begin{eqnarray*}
\varphi_{s_i}^{k+1}(s,0)=...=\varphi_{s_i}^{n}(s,0)=0 \;\;\forall \;i=1,...,k,
\end{eqnarray*}
where $(\varphi_{s_i}^{1}(s,0),...,\varphi_{s_i}^{n}(s,0))$ are the components of $\varphi_{s_i}(s,0)$. By (\ref{Eq:PDE}) and (\ref{Eq:F}), one obtains the following system of PDE:
\[
\left( \begin{array}{c}
\frac{d}{ds_j}\varphi_{s_1}^m \\
\vdots                        \\
\frac{d}{ds_j}\varphi_{s_k}^m 
\end{array} \right)
=\displaystyle{\left(\mu_{ijl}\right)_{1\le i,l\le k}}\cdot
\left( \begin{array}{c}
\varphi_{s_1}^m \\
\vdots                        \\
\varphi_{s_k}^m 
\end{array} \right),
\;\;\;\; \varphi_{s_i}^m(0,0)=0,
\]
for each $j=1,...,k$ and $m=k+1,...,n$.

Denoting by $\xi_i(s_j)=\varphi_{s_i}^m(0,...,0,s_j,0,...,0)$  we obtain the following homogeneous system of ODE:
\[
\left( \begin{array}{c}
\xi '_1 \\
\vdots                        \\
\xi '_k 
\end{array} \right)
=\displaystyle{\left(\mu_{ijl}\right)_{1\le i,l\le k}}\cdot
\left( \begin{array}{c}
\xi_1 \\
\vdots                        \\
\xi_k 
\end{array} \right),
\;\;\;\; \xi_i(0)=0.
\]
This system has only the trivial solution by uniqueness. Therefore $\varphi_{s_i}^{m}(s,0)$ vanish on all coordinate axes for every $i=1,...,k$ and $m=k+1,...,n$. Since we can consider the initial condition at an arbitrary point of one of the coordinate axes, an induction argument on $k$ shows that $\varphi_{s_i}^{k+1}(s,0)=...=\varphi_{s_i}^{n}(s,0)=0$.
\end{proof}

\section{Fibers of the Logarithmic Polar Map}
Let $M$ be a complex variety and $1\le k< \dim(M)$. A \textit {singular holomorphic foliation $\mathcal F$ of dimension $k$} on $M$ is
determined by a line bundle $\mathcal L$ and an element $\omega \in
\mathrm H^0(M, \Omega^{n-k}_{M} \otimes \mathcal L)$ satisfying
\begin{enumerate}
\item[(i)] $\mathrm{codim}(\sing(\omega)) \ge 2$ where $\sing(\omega)
= \{ x \in M \, \vert \, \omega(x) = 0 \}$;
\item[(ii)] $\omega$ is integrable.
\end{enumerate}
By definition $\omega$ is integrable if and only if for every point $x\in M\backslash \sing(\omega)$ there exist a neighborhood $V\subset M$ of $x$ and  $1$-forms $\alpha_1$,...,$\alpha_{n-k} \in \Omega^1_M(V)$ such that
\begin{eqnarray*}
\omega|_V=\alpha_1\wedge \cdots \wedge\alpha_{n-k} \quad \text{and} \quad d\alpha_i \wedge \omega|_V=0 \quad \forall\; i=1,...,n-k.
\end{eqnarray*}

The \textit{singular set} of $\mathcal F$ is by definition equal to  $\sing(\omega)$. The integrability condition (ii) determines in an analytic neighborhood of every regular point $x$, i.e.,
$x \in M \setminus \sing(\omega)$ a holomorphic fibration with
relative tangent sheaf coinciding with the subsheaf of $TM$
determined by the kernel of $\omega$. Analytic continuation of the fibers of this fibration describes the \textit{leaves} of $\mathcal F$.

Let $G$ be a complex algebraic Lie group and $\mathfrak{g}$ its Lie algebra. Let $\mathcal F$ be a singular holomorphic foliation of dimension $k$ on $G$. We shall consider the left Gauss map associated to $\mathcal F$. 

For each $x\in G$ let 
\begin{eqnarray*}
l_x:G &\longrightarrow& G \\
  \tilde{x} &\mapsto& x\tilde{x}
\end{eqnarray*}
be the left translation by $x$ and denote by $\textrm{G} (k,\mathfrak g)$ the grassmannian of $k$--dimensional linear subspaces of $\mathfrak{g}$. For each $x\in G$ regular point of $\mathcal F$, we have a well defined $k$--dimensional linear subspace $T_x\mathcal F\subset T_xG$ which is the tangent space of the leaf of $\mathcal F$ through $x$. Thus the \textit{left Gauss map associated to $\mathcal F$} is the following rational map 
\begin{eqnarray*}
{\mathcal G}(\mathcal F): G &\dashrightarrow& \textrm{G} (k, \mathfrak g) \\
 x &\mapsto& d(l_{x^{-1}})(x)(T_x\mathcal F)
\end{eqnarray*}
defined only on the set of regular points of $\mathcal F$. This is similar to the left Gauss map defined by Kapranov in  \cite{kapranov} for a hypersurface $Z\subset G$.

\begin{example}\label{Ex:gauss}\rm
Let $f \in S_{n}$ be a Laurent polynomial on $n$ variables and $G=(\mathbb C^*)^n$. We consider its zero locus $Z=\{x\in(\mathbb C^*)^n \;|\; f(x)=0\}$. The tangent space of $Z$ at $x$ is given by $T_xZ=\{V\in\mathbb C^n\;|\; f_{x_1}(x)(v_1-x_1)+\cdots+f_{x_n}(x)(v_n-x_n)=0\}$. Taking the left translation to unity of $(\mathbb C^*)^n$ one obtains 
\begin{eqnarray*}
d(l_{x^{-1}})(x)(T_xZ)=\{W\in\mathbb C^n\;|\; x_1f_{x_1}(x)(w_1-1)+\cdots+x_nf_{x_n}(x)(w_n-1)\}.
\end{eqnarray*}
Therefore $d(l_{x^{-1}})(x)(T_xZ)=(x_1f_{x_1}(x):\cdots:x_nf_{x_n}(x))\in\textrm{G} (n-1, \mathfrak g)\cong\mathbb P^{n-1}$. 

Let $\mathcal F$ be  the foliation of dimension $n-1$ on $(\mathbb C^*)^n$ which has as leaves, fibers of the regular function determined by $f$. With the above argument on each leaf we have that the left Gauss map of $\mathcal F$ is just
\begin{eqnarray*}
{\mathcal G}(\mathcal F): (\mathbb C^*)^n &\dashrightarrow& \mathbb P^{n-1} \\
 x &\mapsto& (x_1f_{x_1}(x):\cdots:x_nf_{x_n}(x)).
\end{eqnarray*}
\end{example}

\bigskip

From now on we consider the case $G=(\mathbb C^*)^n$. The left Gauss map for foliations on $(\mathbb C^*)^n$ will be called the \textit{logarithmic Gauss map}. Let $f \in S_{n}$ be a Laurent polynomial on $n$ variables. Suppose the logarithmic polar map $L_f$ has rank $n-k$, $1\le k \le n-1$.  By the implicit function theorem, there is a singular holomorphic foliation $\mathcal F_f$ on $(\mathbb C^*)^n$ of dimension $k$ on which $L_f$ is constant on each leaf. We shall refer to $\mathcal F_f$ as the \textit{foliation determined by fibers} of $L_f$. If $L_f$ has rank $n-k=0$, that is $L_f$ is constant, we still denote by $\mathcal G(\mathcal F_f)$ the constant map taking a point $x\in (\mathbb C^*)^n$ to the Lie algebra $\mathfrak g$. This odd remark will be useful in the proposition below.

\begin{proposition}\label{P:Hesse}
Let $f\in S_{n}$ be a Laurent polynomial. Suppose $f$ has vanishing logarithmic Hessian. Then the logarithmic Gauss map $\mathcal G(\mathcal F_f)$  is constant. 
\end{proposition}

\begin{proof}
We consider the holomorphic map $g:=f\circ\mathcal E$, where $\mathcal E$ is the following recovering of $(\mathbb C^*)^{n}$
\begin{eqnarray*}
\mathcal E:\mathbb C^{n} &\longrightarrow& (\mathbb C^*)^{n}\\
        y    &\mapsto&  \left({\rm{exp}}(y_1),...,{\rm{exp}}(y_n)\right).
\end{eqnarray*}
We have a commutative diagram
\[
 \xymatrix{
 \mathbb C^{n}  \ar[d]_{\mathcal E}  \ar@/^0.4cm/[dr]^{\nabla_g}   \\
  (\mathbb C^*)^{n} \ar@{->}[r]^{L_f} & {\mathbb C}^n}
\]
where $\nabla_g$ is the affine polar map associated to $g$.

We set $x \in (\mathbb C^*)^n$ such that $dL_f(x)$ has maximal rank $n-k$, $1\le k \le n$. If $k=n$, there is nothing to do. Let us suppose $1\le k \le n-1$. Let $E$ be an irreducible component of $L_f^{-1}(L_f(x))$.  It follows from Proposition \ref{T:linearfiber} that $E=\mathcal E (F)$ where $F$ is an affine linear subspace of $\mathbb C^n$. Then $E$ is a left coset of an irreducible algebraic subgroup in $(\mathbb C^*)^n$. But any irreducible algebraic subgroup $H$ of $(\mathbb C^*)^n$ can be written in the form
\begin{eqnarray*}
H=H_{\Lambda}=\{x\in (\mathbb C^*)^n \;|\;\; x^{\lambda}=1 \;\;\forall \lambda \in \Lambda\},
\end{eqnarray*}
where $\Lambda$ is a primitive lattice of $\mathbb Z^n$ of rank $n-k$. So $E=x_0\cdotp H_{\Lambda}$ for some $H_{\Lambda}$. We notice that $H_{\Lambda}=\mathcal E(F_M)$, $F_M=M\otimes \mathbb C$ where $M$ is the primitive lattice
\begin{eqnarray*}
M=\{(y_1,...,y_n)\in \mathbb Z^n \;|\;\; \lambda_1y_1+...+\lambda_ny_n=0 \;\;\forall (\lambda_1,...,\lambda_n)\in\Lambda\}.
\end{eqnarray*}

Let $T_{y}$ be the translation by $y$ in $\mathbb C^n$. By the equality $l_{x^{-1}}\circ \mathcal E=\mathcal E\circ T_{-y}$ one obtains that $d(l_{x^{-1}})(x)(T_x\mathcal F)=F_M$. In particular we can choose a basis of $d(l_{x^{-1}})(x)(T_x\mathcal F)$ of vectors with integer coordinates. Then the image of $\mathcal G(\mathcal F_f)$ is contained in the subset of $\textrm{G} (k, \mathfrak g)=\mathbb G(k-1,\mathbb P \mathfrak g) \subset \mathbb P^N$ on which the points can be chosen with integer coordinates. Since $\mathcal G(\mathcal F_f)$ is a rational map it must be constant.
\end{proof}

\begin{remark}\label{R:left}\rm
We notice that the condition $\mathcal G(\mathcal F_f)$ to be constant means that the irreducible components of  the fibers of $L_f$ are the left cosets of the same subgroup $H_{\Lambda}$ for one fixed primitive lattice $\Lambda$ of $\mathbb Z^n$.
\end{remark}

\begin{remark}\label{R:fiber}\rm

If $f \in \mathbb C[x_0,...,x_n]$ is a homogeneous polynomial, the linear system $<x_0f_{x_0},...,x_nf_{x_n}>$ defines a rational map -- still denoted by -- $L_f:\mathbb P^{n} \longrightarrow \mathbb P^{n}$. It follows from Proposition \ref{P:Hesse} that, under the hypothesis of vanishing logarithmic Hessian, there are linear vector fields $X_i \in \textrm{H}^0(\mathbb P^n, T\mathbb P^n)$, $i=1,...,k$, in the form
\begin{eqnarray*}
\displaystyle{X_i=\sum_{j=0}^{n}\lambda_{ij}x_j\frac{\partial}{\partial x_j}} \;\;;\;\lambda_{ij}\in\mathbb C
\end{eqnarray*}
which are tangent to the fibers of $L_f$. That is, the fibers of $L_f$ are tangents to a Lie subalgebra of $\textrm{H}^0(\mathbb P^n, T\mathbb P^n)$. For the classical polar map $\nabla_f: \mathbb P^n \longrightarrow \mathbb P^n$ defined just by the derivatives $<f_{x_0},...,f_{x_n}>$ it is known that the closure of a generic fiber is a union of linear spaces (see \cite[Proposition 4.9]{Z2}).

Some results have been obtained also to dominant maps. In \cite{B} was obtained a classification of the homogeneous polynomials in $3$ variables on which $L_f$ is birational. This is a variant of the classification obtained in \cite{dolgachev2000polar} of the homogeneous polynomials in $3$ variables on which $\nabla_f$ is birational.
\end{remark}

\section{Hesse's Problem on the Torus}

Given $A=(a_{ij})\in M(n,\mathbb Z)$ a matrix with integer coefficients,  it defines an endomorphism $\xi_A\in{\rm{End}}(\mathbb C^*)^{n}$
\begin{eqnarray*}
\xi_A(x_1,...,x_n)=\left( x_1^{a_{11}}\cdots x_n^{a_{1n}},...,x_1^{a_{n1}}\cdots x_n^{a_{nn}}\right)
\end{eqnarray*}  
which induces an element $\xi_A^* \in {\rm{End}}(S_{n})$, where ${\rm{End}}(S_{n})$ denotes the set of endomorphisms of the coordinate ring of the torus $(\mathbb C^*)^{n}$. We note that $\xi_{AB}=\xi_B\circ \xi_A$, making it clear that, if $\textrm{det}(A)=\pm 1$, then $\xi_A$ is an automorphism with inverse $\xi_{A^{-1}}$. Any automorphism $\xi$ of $(\mathbb C^*)^{n}$ is of the form $\xi=\xi_A$ for some invertible matrix $A \in M(n,\mathbb Z)$ and it is classically called a \textit{monoidal transformation}.

\begin{theorem}\label{T:Hesse}
Let $f\in S_{n}$ be a Laurent polynomial. If $f$ has vanishing logarithmic Hessian of rank $n-k$, $1\le k\le n$, then there is an automorphism $\xi$ of $(\mathbb C^*)^{n}$ such that $\xi^*f \in S_{n-k}$. 
\end{theorem}

\begin{proof}
We consider the holomorphic map $g=f\circ\mathcal E$, and let $\nabla_g$ its affine polar map. It follows from Proposition \ref{P:Hesse} that the irreducible components of the fibers of $\nabla_g$ are affine linear spaces $F_M + y_0=\{ y + y_0 \;|\; y\in F_M\}$, $y_0 \in \mathbb C^n$,  that is, the left cosets of $F_M=M\otimes \mathbb C$ in $\mathbb C^n$. Where $M=\Lambda^{\perp}$ is a primitive lattice of $\mathbb Z^n$ of rank $k$.

Let $\{Y_{1},...,Y_{n}\}$ be a basis of $\mathbb Z^n$ such that $Y_{1},...,Y_{n-k}$ is a basis of $\Lambda$ and $Y_{n-k+1},...,Y_{n}$ is a basis of $M$. The matrix $A \in M(n,\mathbb Z)$ in which $A(e_i)=Y_i$, $i=1,...,n$, defines a  linear map $A:\mathbb C^{n} \longrightarrow \mathbb C^{n}$ such that $\textrm{det}(A)=\pm 1$.

We have a commutative diagram
\[
 \xymatrix{
\mathbb C^{n} \ar[r]^{A} \ar[d]_{\mathcal E} & \mathbb C^{n}  \ar[d]_{\mathcal E}  \ar@/^0.4cm/[dr]^{g}  \\
(\mathbb C^*)^{n} \ar[r]^{\xi_A} & (\mathbb C^*)^{n} \ar@{->}[r]^{f} & {\mathbb C.}}
\]
If $\tilde{g}=g\circ A$, then the affine polar map $\nabla_{\tilde{g}}$ is independent of the variables $y_{n-k+1},...,y_n$, that is,
\begin{eqnarray}\label{Eq:integration}
\tilde{g}_{y_1},...,\tilde{g}_{y_n}\in S_{n-k}.
\end{eqnarray}

We set $\tilde{f}=\xi_A^*\cdot f\in S_n$ and assume 
\begin{eqnarray*}
\tilde{f}=\sum_{I}a_I\cdot x_1^{i_1}\cdots x_n^{i_n}, \;\;\;a_I\neq 0,
\end{eqnarray*}
with $I=(i_1,...,i_n)\in \mathbb Z^n$. Hence 
\begin{eqnarray*}
\tilde{g}=\sum_{I}a_I\cdot e^{i_1y_1+...+i_ny_n}, \;\;\;a_I\neq 0.
\end{eqnarray*}
Now we take the derivatives 
\begin{eqnarray*}
\tilde{g}_{y_j}=\sum_{I}a_I\cdot i_j\cdot e^{i_1y_1+...+i_ny_n},
\end{eqnarray*}
for $j=n-k+1,...,n$. Fix $j\in\{n-k+1,...,n\}$. By (\ref{Eq:integration}), $\displaystyle{\tilde{g}_{y_j}}$ is independent of $y_j$, therefore $i_j=0$ for all index $I$ (consequently $\displaystyle{\tilde{g}_{y_j}}=0$). This shows that 
\begin{eqnarray*}
\tilde{f}=\sum_{I}a_I\cdot x_1^{i_1}\cdots x_{n-k}^{i_{n-k}} \in S_{n-k}.
\end{eqnarray*}
\end{proof}

\bibliographystyle{plain}
\bibliography{new-hesse}
\end{document}